\documentclass[12pt]{amsart}

\usepackage{amscd} 
\usepackage{amsmath} 
\usepackage{amssymb} 
\usepackage{amsthm}
\usepackage{bigdelim}
\usepackage{color} 
\usepackage{enumerate}
\usepackage{graphicx}
\usepackage{mathrsfs}
\usepackage{multirow}
\usepackage[all]{xy} 
\usepackage[dvipdfmx]{hyperref}
\usepackage{comment}

\newtheorem{thm}{Theorem}[section] 

\newtheorem{cor}[thm]{Corollary}
\newtheorem{prop}[thm]{Proposition}
\newtheorem{conj}[thm]{Conjecture}

\newtheorem{lem}[thm]{Lemma}
\theoremstyle{definition} 
\newtheorem{defn}[thm]{Definition}

\newtheorem{thmdefn}[thm]{Theorem-Definition} 
 
\newtheorem{ex}[thm]{Example} 
\theoremstyle{remark}
\newtheorem{rem}[thm]{Remark}

\newtheorem{ques}[thm]{Question}
\newtheorem{problem}[thm]{Problem}

\newtheorem{claim}{Claim}
\newtheorem{step}{Step}

\newcommand{\Sym}{{\rm Sym}}

\title[Almost nef regular foliations and Fujita's decomposition]
{Almost nef regular foliations  and \\ Fujita's decomposition of reflexive sheaves}

\author{Masataka IWAI}
\address{Osaka City University Advanced Mathematical Institute, Osaka City University, 3-3-138, Sugimoto, Sumiyoshi-ku Osaka, 558-8585
Japan} 
\email{{\tt masataka@sci.osaka-cu.ac.jp, masataka.math@gmail.com}}

\date{\today, version 0.01}


\begin{document}
\maketitle
\begin{abstract}
In this paper, we study almost nef regular foliations.
We give a structure theorem of a smooth projective variety $X$ with an almost nef regular foliation $\mathcal{F}$:
$X$ admits a smooth morphism $f: X \rightarrow Y$ with rationally connected fibers such that $\mathcal{F}$ is a pullback of a numerically flat regular foliation on $Y$.
Moreover, $f$ is characterized as a relative MRC fibration of an algebraic part of $\mathcal{F}$.
As a corollary, an almost nef tangent bundle of a rationally connected variety is generically ample.
For the proof, we generalize Fujita's decomposition theorem.
As a by-product, we show that a reflexive hull of $f_{*}(mK_{X/Y})$ is a direct sum of a hermitian flat vector bundle and a generically ample reflexive sheaf for any algebraic fiber space $f : X \rightarrow Y$.
We also study foliations with nef anti-canonical bundles.

\end{abstract}
\tableofcontents

\section{Introduction}
The structure theorem of smooth projective varieties with nef tangent bundles was established by
Campana-Peternell and Demailly-Peternell-Schneider in \cite{CP} and \cite{DPS} after the Hartshorne conjecture had been solved by Mori in \cite{Mori}.
These works suggest that the structures of algebraic varieties are restricted when the tangent bundles have algebraic positivities.
Recently, these related studies have been extensively investigated by \cite{LP}, \cite{Iwai}, \cite{HIM}, and so on.

On the other hand, 
even if a tangent bundle $T_X$ of a smooth projective variety $X$ 
contains a subsheaf $\mathcal{F} $ with some algebraic positivity, then the structure of  $X$ is expected to be restricted.
In \cite{AW} and \cite{Liu}, 
if $T_X$ contains a rank $r$ ample subsheaf $\mathcal{F}$, 
then $X$ is isomorphic to $\mathbb{CP}^{n}$ and $\mathcal{F}$ is isomorphic to either $T_{\mathbb{CP}^n}$ or $\mathcal{O}_{\mathbb{CP}^n}(1)^{\oplus r}$.
In a recent paper \cite{LOY20}, 
if $T_X$ contains a rank $r$ strictly nef locally free subsehaf $\mathcal{F}$,
then $X$ admits a $\mathbb{CP}^{d}$-bundle structure $f :  X \rightarrow Y$ for some integer $d \ge r$ such that either $\mathcal{F}$ is isomorphic to $ T_{X/Y}$, or $\mathcal{F}$ is numerically projectively flat and 
$\mathcal{F} |_{\mathbb{CP}^{d}} \cong \mathcal{O}_{\mathbb{CP}^d}(1)^{\oplus r}$ on any fiber of $f$.

In \cite{Pet}, Peternell proposed the following problem.
\begin{problem}\cite[Problem 4.18]{Pet}
\label{peternell_question}
Let $\mathcal{F} \subset T_X$ be a locally free (or reflexive) sheaf.
What can be said on the structure of $X$ if $\mathcal{F}$ is nef, pseudo-effective, or almost nef? 
\end{problem}

In this paper, we give a partial answer to Problem \ref{peternell_question}.

\begin{thm}[=Theorem \ref{psefregular} and Corollary \ref{ampleregular}]
\label{psef_thm}

Let $X$ be a smooth $n$-dimensional projective variety and 
$\mathcal{F} \subset T_X$ be a rank $r$ almost nef regular foliation with $c_1(\mathcal{F}) \neq 0$.
Then there exist a smooth morphism $f : X \rightarrow Y$ between smooth projective varieties and a numerically flat regular foliation $\mathcal{G} \subset T_Y$ on $Y$, such that all fibers of $f$ are rationally connected, 
$\mathcal{F}=f^{-1}\mathcal{G}$, and there is an exact sequence
\begin{equation*}
\xymatrix@C=25pt@R=20pt{
0\ar@{->}[r]&T_{X/Y} \ar@{->}[r]&\mathcal{F} \ar@{->}[r]
 & f^{*}\mathcal{G} \ar@{->}[r]&0. \\}
\end{equation*}

Moreover, the following statements hold.
\begin{enumerate}
\item If $\mathcal{F}$ is ample, then $X$ is isomorphic to $\mathbb{CP}^{n}$. 
\item If $\mathcal{F}$ is nef, then all fibers of $f$ are Fano.
\item If $\mathcal{F}$ is V-big, then $\mathcal{F} =T_{X/Y}$ and all fibers of $f$ are isomorphic to $\mathbb{CP}^{r}$.
\item If $\mathcal{F}$ is nef and V-big, then $X$ is isomorphic to $\mathbb{CP}^{n}$. 
\end{enumerate}

\end{thm}
Furthermore, $f$ is characterized as a relative MRC fibration of an algebraic part of $\mathcal{F}$ (see Theorem \ref{posiotive_mrc}).
We point out that a foliation $\mathcal{F}$ is regular if $\mathcal{F}$ is a subbundle of $T_X$, and that $f^{-1}\mathcal{G}$ is a pull-back of $\mathcal{G}$ via $f$ (see Chapter 2.2).
Notice that (1) in Theorem \ref{psef_thm} has already been proved by \cite{AW} and \cite{Liu}.
In \cite[Theorem 3.9]{Liu}, if a subsheaf $\mathcal{F} \subset T_X$ is ample, then a saturation of $\mathcal{F}$ in $T_X$ is an algebraically integrable foliation.

In \cite[Theorem 1.8]{HP19}, if a reflexive sheaf $\mathcal{F}$ is almost nef and $c_1(\mathcal{F}) = 0$, then $\mathcal{F}$ is numerically flat. 
Therefore the classifications of almost nef regular foliations
can be reduced to the classifications of regular foliations with vanishing Chern classes.
These foliations were deeply studied
in \cite{PT13}, \cite{Druel16}, and so on.
In Section \ref{Classification}, we classify almost nef regular foliations on smooth projective surfaces.

We have the following corollary.
This is a partial answer to a conjecture by Peternell in \cite[Conjecture 4.12]{Pet1}.

\begin{cor}[=Corollary \ref{psef_example}]
\label{almost_nef_tangent}
Let $X$ be a smooth $n$-dimensional projective variety and
$\mathcal{F} \subset T_X$ be a foliation on $X$.
Assume that $T_X$ is almost nef.
Then $\mathcal{F}$ is induced by a smooth MRC morphism if and only if $\mu_{\alpha}^{min}(\mathcal{F}) > 0$ and $K_{\mathcal{F}} \equiv K_X$, where $\alpha = H_1 \cdots H_{n-1}$ for some ample divisors $ H_1, \ldots, H_{n-1}$ on $X$.

In particular, if $T_X$ is almost nef and $X$ is rationally connected, then $T_X$ is $(H_1, \ldots, H_{n-1})$-generically ample for any ample divisors $H_1, \ldots, H_{n-1}$ on $X$.
\end{cor}
By \cite{DPS01}, \cite{HP19}, and the argument of \cite[Theorem 1.1]{HIM}, if $T_X$ is almost nef, 
then we can take an MRC fibration of $X$ as a smooth morphism.
By Corollary \ref{almost_nef_tangent}, an MRC fibration is characterized by slopes and canonical divisors. 

For the proof of Theorem \ref{psef_thm}, we generalize Fujita's decomposition theorem in \cite{Fujita} and \cite{CD}.
The original Fujita's decomposition theorem states that, for any algebraic fiber space $f : X \rightarrow C$ to a smooth projective curve, $f_{*}(K_{X/C})$ is a direct sum of a unitary flat vector bundle and an ample vector bundle.
In \cite{CK}, Fujita's decomposition theorem over higher dimensional base was studied.
We establish a generalization of Fujita's decomposition theorem for almost nef vector bundles and reflexive coherent sheaves which have positively curved singular hermitian metrics.
As a by-product, we get the following theorem.
\begin{thm}[=Corollary \ref{Fujita_decomposition}]
\label{Fujita_thm}
Let $f : X \rightarrow Y$ be a surjective 
morphism with connected fibers from a compact K\"ahler manifold $X$ to a smooth $d$-dimensional projective variety $Y$.
For any positive integer $m$,  
there is a unique decomposition $$( f_{*}(mK_{X/Y}) )^{\vee\vee} \cong Q \oplus G,$$
where $Q$ is a hermitian flat vector bundle and $G$ is an $( H_1, \ldots, H_{d -1})$-generically ample reflexive coherent sheaf for any ample divisors $H_1, \ldots, H_{d -1}$ on $Y$.

In particular, if $Y$ is a curve, 
for any positive integer $m$, 
there is a unique decomposition $$ f_{*}(mK_{X/Y})  \cong Q \oplus G,$$
where $Q$ is a hermitian flat vector bundle and $G$ is an ample vector bundle.

\end{thm}
Notice that $( f_{*}(mK_{X/Y}) )^{\vee\vee}$ is a reflexive hull of $f_{*}(mK_{X/Y}) $.
Theorem \ref{Fujita_thm} is a generalization of Fujita's decomposition theorem in \cite{Fujita}, \cite{CD}, and \cite{CK}.

Finally,  we study foliations with nef anti-canonical bundles.
In \cite{Druel}, these foliations were studied in detail. 
Druel proposed the following question in \cite{Druel}.
\begin{ques} \cite[Question 7.4]{Druel}
\label{druel_conjecture}
Let $\mathcal{F}$ be a regular foliation on $X$.
If $-K_{\mathcal{F}}$ is nef, then do we have $\kappa(X, -K_{\mathcal{F}}) \le {\rm rk} \mathcal{F} $?
\end{ques}
In \cite[Theorem 8.8]{Druel}, if $-K_{\mathcal{F}}$ is nef and abundant, then $\kappa(X, -K_{\mathcal{F}}) \le {\rm rk} \mathcal{F} $ 
and equality holds only if $\mathcal{F}$ is algebraically integrable. We give an answer to Question \ref{druel_conjecture}. 

\begin{thm}[= Theorem \ref{rcfoliation}]
\label{rcfoliation_main}
Let $X$ be a smooth projective variety and $\mathcal{F} \subset T_X$ be a foliation on $X$.
Assume that $\mathcal{F}$ is regular or $\mathcal{F}$ has a compact leaf. 
If $-K_{\mathcal{F}}$ is nef, then $\kappa(X, -K_{\mathcal{F}}) \le {\rm rk} \mathcal{F} $.

Moreover, when the equality holds, the followings hold.
\begin{enumerate}
\item $\mathcal{F}$ is an algebraically integrable foliation with rationally connected leaves.
\item $-K_{F}$ is semiample and big for a general leaf $F$ of $\mathcal{F}$.
\item $-K_{\mathcal{F}}$ is semiample and abundant.
\end{enumerate}
\end{thm}

{\bf Acknowledgment.} 
The author wishes to express his thanks to Prof. Shin-ichi Matsumura and Prof. Takayuki Koike for some discussions and helpful comments. 
He would like to thank Prof. St\'ephane Druel for answering his questions and for some helpful comments.
He also would like to thank Prof. Takeo Ohsawa for some useful advice.
The author wishes to express his thanks to anonymous referees for pointing out his mistakes.
The proof of Theorem \ref{posiotive_mrc} in the previous version 
is improved by their comments. 
The author was supported by the public interest incorporated foundation Fujukai and by Foundation of Research Fellows, The Mathematical Society of Japan.
This work was partly supported by Osaka City University Advanced Mathematical Institute (MEXT Joint Usage/Research Center on Mathematics and Theoretical Physics JPMXP0619217849).

\section{Preliminaries}
Throughout this paper, we work over the field $\mathbb{C}$ of complex numbers.
We denote by $\mathbb{N}_{>0}$ the set of positive integers
and denote $\mathcal{H}om(\mathcal{F}, \mathcal{O}_{X})$ 
by $\mathcal{F}^{\vee}$
for any torsion-free coherent sheaf $\mathcal{F}$ on any variety $X$.
We interchangeably use the words “Cartier divisors” 
and “line bundles”. 

In this chapter, we review some of the standard facts on foliations, slopes, and so on.
For more details, we refer the reader
to  \cite[Chapter 1]{Clauden}, \cite[Chapter 2 and 4]{Lazic}, and \cite[Appendices B]{Wang}.
\subsection{Foliations}
\text{}

Let $X$ be a normal variety and $T_X$ be a tangent sheaf of $X$. 
A \textit{$($singular$)$ foliation} is a saturated subsheaf $\mathcal{F} \subset T_X$ which is closed under the Lie bracket.
A foliation $\mathcal{F}$ is called \textit{regular} if $X$ is smooth and $\mathcal{F} $ is a subbundle of $T_X$.
If $X$ is projective, the \textit{canonical divisor} $K_{\mathcal{F}} $ of $\mathcal{F}$ is defined by a divisor such that $\mathcal{O}_{X}(K_{\mathcal{F}} ) \cong \det(\mathcal{F})^{\vee}$.
Let $X_{\mathcal{F}}$ denote the subset of the regular locus $X_{reg}$ on which $\mathcal{F} $ is a subbundle of $T_X$. 
A \textit{leaf} of $\mathcal{F}$ is the maximal connected, locally closed submanifold $L \subset X_{\mathcal{F}}$ such that $T_{L} = \mathcal{F}|_{L}$.
 A leaf $L$ of $\mathcal{F}$ is \textit{algebraic} if it is open in its Zariski closure $\overline{L}^{zar}$ and $\dim L = \dim \overline{L}^{zar}$.
A foliation $\mathcal{F}$ is called \textit{algebraically integrable}  if every leaf passing through a general point of $X$ is algebraic.

\begin{ex}
Let $f : X \rightarrow Y$ be a surjective morphism with connected fibers between normal projective varieties.
Then the kernel of the differential map $df : T_{X} \rightarrow f^{*}T_Y$
defines a foliation $\mathcal{F}$ on $X$. 
$\mathcal{F}$ is an algebraically integrable foliation and a general fiber of $f$ is a leaf of $\mathcal{F}$.
We say that \textit{$\mathcal{F}$ is induced by $f$.}
Similarly, a dominant rational map induces an algebraically integrable foliation.

\end{ex}

By the following lemma, an algebraically integrable foliation is induced by a dominant rational map.

\begin{lem}[{\cite[Section 2.13]{Druel}, \cite[Proposition B.19]{Wang}}]
\label{algebraic}
Let $X$ be a normal projective variety and 
$\mathcal{F}$ be an algebraically integrable foliation on $X$.
Then there exist a unique normal projective variety $Y$ contained in the normalization of the Chow variety of $X$, 
a projective variety $Z$ which is the normalization of the universal cycle, 
a birational morphism $\pi : Z \rightarrow X$, an equidimensional dominant morphism $f : Z \rightarrow Y$ with connected fibers, and an effective $\mathbb{Q}$-divisor $B$ on $Z$, such that 
 $\pi(f^{-1}(y)) \subset X$ is a Zariski closure of a leaf of $\mathcal{F}$ for a general point $y \in Y$ and $\pi^{*}K_{\mathcal{F}} \sim_{\mathbb{Q}} K_{\widetilde{\mathcal{F}}}+B$, where $\widetilde{\mathcal{F}}$ is the foliation induced by $f$.
 
\begin{equation*}
\xymatrix@C=25pt@R=20pt{
Z  \ar@{->}[d]_{\pi}  \ar@{->}[r]^{f}&  Y \\   
X& \\
}
\end{equation*}
In particular, $\widetilde{\mathcal{F}}$ coincides with $\pi^{*}\mathcal{F}$ on some Zariski open subset of $Z$.
\end{lem}

\begin{defn}[{\cite[Notation 2.7]{Druel}, \cite[Remark B.18]{Wang}}]
Let $f : X \rightarrow Y$ be a surjective morphism with connected fibers between smooth projective varieties, or an equidimensional dominant morphism with connected fibers between normal projective varieties.  
The \textit{ramification divisor of $f$} is defined by
$$Ram(f) := \sum_{D \subset Y} (f^{*}D - (f^{*}D)_{red}),$$
where $D$ runs through all prime divisors on $Y$.
\end{defn}

\begin{ex}[{\cite[Section 2.9]{Druel}, \cite[Remark B.18]{Wang}}]
\label{ramification_algebraic}
Let $f : X \rightarrow Y$ be an equidimensional dominant morphism with connected fibers between normal projective varieties and  $\mathcal{F}$ be a foliation induced by $f$.
Then
$$K_{\mathcal{F}} \sim_{\mathbb{Q}} K_{X/Y} - Ram(f)
.$$

\end{ex}

\begin{lem}
\label{ramification}
Let $X,Y,Z$ be smooth projective varieties and 
$f: X \rightarrow Y$, $g: Z \rightarrow Y$, $h: X \rightarrow Z$ be surjective equidimensional morphisms with connected fibers such that $f = g \circ h$.
\begin{equation*}
\xymatrix@C=25pt@R=20pt{
X \ar@{->}[dr]_{f}  \ar@{->}[r]^{h} &  Z  \ar@{->}[d]^{g} \\   
&Y \\
}
\end{equation*}
Then $Ram(h) + h^{*}Ram(g) - Ram(f)$ is effective.
\end{lem}
\begin{proof}

Let $\mathcal{I}$ denote the set of prime divisors $E \subset Z$ such that $g(E)$ is a divisor on $Y$. Set $R :=\sum_{E \in \mathcal{I}} h^{*}E - (h^{*}E)_{red} $, then $Ram(h) - R $ is effective.
We show that $R=Ram(f)- h^{*}Ram(g).$
For any $E \in \mathcal{I}$, write $D:=g(E)$.
Let $w \in \mathbb{N}_{>0}$ be the coefficient of $E$ in $g^{*}D$.
Set $h^{*}E = \sum d_{i}F_{i}$ for some $d_i \in \mathbb{N}_{>0}$ and some prime divisors $F_i$ on $X$.
Since the coefficient of $Ram(f)- h^{*}Ram(g)$ in $F_{i}$ is 
$(d_{i} w-1)-d_{i}(w-1) = d_{i} - 1$,
$Ram(f)- h^{*}(Ram(g)) \ge \sum (d_{i}-1 )F_{i}$.
Hence $Ram(f)- h^{*}Ram(g) -R  \ge 0$.
By a similar argument, $R- Ram(f) + h^{*}Ram(g) \ge 0$.
Therefore $Ram(h) + h^{*}Ram(g) - Ram(f) = Ram(h) - R $ is effective.
\end{proof}

By Lemma \ref{algebraic}, Example \ref{ramification_algebraic}, and the flattening theorem of \cite{Raynord},
we obtain the following proposition.

\begin{prop}\cite[Proposition 1.4]{Clauden}
\label{raynord_foliation}
Let $X$ be a smooth projective variety and 
$\mathcal{F}$ be an algebraically integrable foliation on $X$.
Then there exist a birational morphism $\pi : \widetilde{X} \rightarrow X$,
a surjective morphism $\widetilde{f} : \widetilde{X} \rightarrow Y$ with connected fibers between smooth projective varieties, and $\pi$-exceptional divisors $E, E'$ on $\widetilde{X}$, such that any $\tilde{f}$-exceptional divisor is $\pi$-exceptional and
$$
\pi^{*}K_{\mathcal{F}} \sim_{\mathbb{Q}} K_{\widetilde{X}/Y} -Ram(\widetilde{f}) + E \sim_{\mathbb{Q}} K_{\widetilde{\mathcal{F}}} +E',
$$
where $\widetilde{\mathcal{F}}$ is a foliation induced by $\widetilde{f}$.
\end{prop}

\subsection{Algebraic parts and transcendental parts of foliations}
\text{}


First we define a pull-back of a foliation via a dominant rational map.
For more details, we refer the reader to \cite[Section 3.1, 3.2]{Druel18} and \cite[Proposition-Definition B.3, B.4]{Wang}.

Let $X$ be a normal $n$-dimensional variety and $\mathcal{F}$ be a rank $r$ foliation on $X$.
Set $N_{\mathcal{F} } := (T_X/\mathcal{F})^{\vee\vee}$.
From $N_{\mathcal{F} }^{\vee} \subset (\Omega_{X}^{1})^{\vee\vee}$, we obtain a $\det N_{\mathcal{F} }$-valued $n-r$ differential form
$\omega_{\mathcal{F}} \in H^{0}(X, (\Omega_{X}^{n-r}\otimes \det N_{\mathcal{F} } )^{\vee\vee})$.
$\omega_{\mathcal{F}}$ satisfies the following three conditions.
\begin{enumerate}
\renewcommand{\theenumi}{\Alph{enumi}}
\label{pfaff}
\item A codimension of a vanishing locus of $\omega_{\mathcal{F}}$ is at least two.
\item $\omega_{\mathcal{F}}$ is locally decomposable, that is, 
in a neighborhood $U$ of a general point of $X$, we have $\omega_{\mathcal{F}}|_{U} = \omega_1 \wedge \cdots \wedge \omega_{n-r}$ for some 1-forms $\omega_{1}, \cdots , \omega_{n-r}$ on $U$.
\item $\omega_{\mathcal{F}}$ is integrable, that is, $d\omega_{i} \wedge \omega_{\mathcal{F}} =0$ for any $1 \le i \le n-r$ under the condition (B).
\end{enumerate}

Conversely, for any rank 1 reflexive sheaf $\mathcal{L}$ on $X$ and 
$\omega \in H^0(X,(\Omega_{X}^{n-r}\otimes \mathcal{L})^{\vee\vee})$ which satisfies three conditions (A)-(C) above, a rank $r$ foliation is induced by
the kernel of the morphism $T_X \rightarrow (\Omega_{X}^{n-r-1}\otimes \mathcal{L})^{\vee\vee}$ given by the contraction of $\omega$.

Let $f : X \dashrightarrow Y$ be a dominant rational map to a normal $m$-dimensional  variety and  $\mathcal{G}$ be a rank $l$ foliation on $Y$.
Then there exist smooth Zariski open sets $X_0 \subset X$ and $Y_0 \subset Y$ such that 
$f_0 := f|_{X_0}: X_0  \rightarrow Y_0$ is a morphism.
By the above argument, we can take $\omega_{\mathcal{G}} \in H^{0}(Y, (\Omega_{Y}^{m-l}\otimes \det N_{\mathcal{G}})^{\vee\vee} )$. 
By shrinking $X_0$ and $Y_0$,
$\omega := df_{0}( \omega_{\mathcal{G}} |_{Y_0})$ 
satisfies the three conditions (A)-(C) above, 
thus $\omega$ induce a foliation $\mathcal{F}_0$ on $X_0$.
The \textit{pull-back $f^{-1}\mathcal{G}$ of $\mathcal{G}$ via $f$} is the foliation on $X$ whose restriction to $X_0$ is $\mathcal{F}_0$.

\begin{rem}
Let $X$ be a normal $n$-dimensional variety and 
$\mathcal{F}$ be a rank $r$ foliation on $X$.
By Frobenius Theorem, for any $x \in X_{\mathcal{F}}$, 
there exist a neighborhood  $U  \subset X_{\mathcal{F}}$ of $x$ and 
local generators $v_1, \ldots, v_n$ of $T_X$ on $U$
such that $\mathcal{F}|_{U}$ is generated by $v_1, \ldots, v_r$.
Then  
$\omega_{\mathcal{F}}|_{U} $ is identified with $ \alpha_{r+1} \wedge \cdots \wedge \alpha_{n}$, where $\alpha_{i}$ is a dual of  $v_i$ for $ r+1 \le i \le n$. 
\end{rem}

\begin{lem}[{\textit{cf.} \cite[Lemma 6.7]{AD14}}]
\label{decent}
Let $f : X \rightarrow Y$ be a smooth morphism with connected fibers between smooth projective varieties and $\mathcal{F}$ be a rank $r$ regular foliation on $X$. 
Assume that all fibers of $f$ are contained in the leaves of $\mathcal{F}$.
Then there exists a regular foliation $\mathcal{G}$ on $Y$ such that 
$\mathcal{F}=f^{-1}\mathcal{G}$ and 
\begin{equation*}
\xymatrix@C=25pt@R=20pt{
0\ar@{->}[r]& T_{X/Y}  \ar@{->}[r]&\mathcal{F}\ar@{->}[r]
 &f^{*}\mathcal{G}\ar@{->}[r]&0 .\\}
\end{equation*}
\end{lem}
\begin{proof}

Set $\dim Y := m$, $k := \dim X - \dim Y $, and $n := \dim X = m+k $.
Fix $y \in Y$.
Let $V$ be a neighborhood of $y$ and 
$(w_{1}, \ldots, w_{m})$ be  local coordinates on $V$.
We may regard $y \in V$ as an origin.
We can take Euclidean open sets $U_{1}, \ldots, U_{N}$ of $X$
which satisfy the following conditions:
\begin{enumerate}
\item $X_y \subset \cup_{\lambda=1}^{N} U_{\lambda}$.
\item For any $1 \le \lambda \le N$, we can choose local coordinates $(z_{1}^{\lambda}, \ldots, z_{n}^{\lambda})$ on $U_{\lambda}$ such that
$$
\begin{array}{cccc}
f|_{U_{\lambda} } :&  U_{\lambda}    & \rightarrow& V       \\
&(z_{1}^{\lambda}, \ldots, z_{n}^{\lambda})                &    \mapsto  & (z_{1}^{\lambda}, \ldots, z_{m}^{\lambda}).
\end{array}
$$
In particular, $T_{X/Y} |_{U_{\lambda}}$ is generated by $\frac{\partial}{\partial z_{m+1}^{\lambda}}, \ldots, \frac{\partial}{\partial z_{n}^{\lambda}}$ on $U_{\lambda}$.
\item For any $1 \le \lambda \le N$, 
there exists a holomorphic morphism $p_{\lambda}$ on $U_{\lambda}$ such that 
$$
\begin{array}{cccc}
p_{\lambda}  :&  U_{\lambda}    & \rightarrow& \mathbb{C}^{n-r}      \\
&(z_{1}^{\lambda}, \ldots, z_{n}^{\lambda})                &    \mapsto  & (z_{1}^{\lambda}, \ldots,  z_{n-r}^{\lambda})
\end{array}
$$
and $\ker (dp_{\lambda}) = \mathcal{F}|_{U_{\lambda}}$. 
In particular, $\mathcal{F}|_{U_{\lambda}}$ is generated by
$\frac{\partial}{\partial z_{n-r+1}^{\lambda}}, \ldots, \frac{\partial}{\partial z_{n}^{\lambda}}$ on $U_{\lambda}$.
\end{enumerate}
Then there exists a local section  $s_{\lambda} : V \rightarrow U_{\lambda}$ such that
$$\begin{array}{cccc}s_{\lambda}  :&  V & \rightarrow& U_{\lambda}      \\&(w_{1}, \ldots, w_{m})                &    \mapsto  & (w_{1}, \ldots, w_{m}, 0, \ldots, 0)\end{array}
$$
for any $1 \le \lambda \le N$.
Set $\mathcal{G}_{V} := \ker (d(p_{\lambda}\circ s_{\lambda}))$, then $\mathcal{G}_{V}$ does not depend on  $\lambda$, and $\mathcal{G}_{V}$ glue and extend to an $r-k$ rank foliation $\mathcal{G}$ on $Y$
(see \cite[Lemma 6.7]{AD14}).
$\mathcal{G}_{V}$ 
is generated by $\frac{\partial}{\partial w_{n-r+1}}, \ldots, \frac{\partial}{\partial w_{m}}$ on $V$. 
Therefore $\mathcal{G}|_{V}$ is a subbundle of $T_V$, 
$f^{*} \mathcal{G}|_{U_{\lambda}}$ is locally free, and
\begin{equation*}
\xymatrix@C=25pt@R=20pt{
0\ar@{->}[r]& T_{X/Y}|_{U_{\lambda}}  \ar@{->}[r]& \mathcal{F} |_{U_{\lambda}}\ar@{->}[r]
 &f^{*}\mathcal{G}|_{U_{\lambda}}\ar@{->}[r]&0, \\}
\end{equation*}
for any $1 \le \lambda \le N$.
From $\omega_{\mathcal{G}} |_{V} = dw_{1} \wedge \cdots \wedge dw_{n-r}$ and $\omega_{\mathcal{F}} |_{U_{\lambda}} = dz_{1}^{\lambda} \wedge \cdots \wedge dz_{n-r}^{\lambda}$,
we obtain $\mathcal{F}|_{U_{\lambda} } = f^{-1}\mathcal{G}|_{U_{\lambda} } $.
\end{proof}
\begin{lem}[{\cite[Lemma 6.7]{AD14}}]
\label{rationaldecent}
Let $f: X \dashrightarrow Y$ be an essentially equidimensional \footnote{A rational map $f: X \dashrightarrow Y$ is \textit{essentially equidimensional} if there exists a Zariski open set $U \subset X$ such that ${\rm codim}(X \setminus U) \ge 2$ and $f|_{U}$ is an equidimensional morphism. We refer the reader to \cite[Definition 2.22]{CKT16}.}
 dominant rational map with connected fibers between normal projective varieties, 
 $\mathcal{H}$ be a foliation induced by $f$, and $\mathcal{F}$ be a foliation on $X$.  
Assume that a general fiber of $f$ is contained in a leaf of $\mathcal{F}$.
Then there exists a foliation $\mathcal{G}$ on $Y$ such that 
$\mathcal{F}=f^{-1}\mathcal{G}$ and 
\begin{equation*}
\xymatrix@C=25pt@R=20pt{
0\ar@{->}[r]& \mathcal{H}  \ar@{->}[r]& \mathcal{F}\ar@{->}[r]
&(f^{*}\mathcal{G})^{\vee\vee}. \\}
\end{equation*}
\end{lem}
\begin{proof}
We take Zariski open sets $X_1 \subset X_{\mathcal{F}}$ and $Y_1 \subset Y_{reg}$ such that $\mathcal{H}|_{X_1}$ is a subbundle of $ \mathcal{F}|_{X_1}$ and $f|_{X_1}: X_1 \rightarrow Y_1$ is a morphism.
We may assume that ${\rm codim}(X \setminus X_1) \ge 2$.
Let $X_0 \subset X_1$ be the set of smooth points of $f$ and  $Y_0 :=f(X_0)$.
By the argument of Lemma \ref{decent}, there exists a foliation $\mathcal{G}_{0}$ on $Y_0$ such that $\mathcal{F}|_{X_0} = (f|_{X_0})^{-1}\mathcal{G}_{0}$ and 
\begin{equation*}
\xymatrix@C=25pt@R=20pt{
0\ar@{->}[r]& \mathcal{H}|_{X_0}  \ar@{->}[r]&  \mathcal{F} |_{X_0}\ar@{->}[r]
 &(f|_{X_0})^{*}\mathcal{G}_0|_{X_0} \ar@{->}[r]&0 . \\}
\end{equation*}

Let $\mathcal{G}$ denote a saturation of $\mathcal{G}_{0}$ in $T_Y$. 
By \cite[Lemma B.2]{Wang} and the argument of \cite[Lemma 6.7]{AD14},
$\mathcal{G}$ is a foliation on $Y$, 
$\mathcal{F} = {f}^{-1}\mathcal{G}$, and
\begin{equation*}
\xymatrix@C=25pt@R=20pt{
0\ar@{->}[r]& \mathcal{H}|_{X_1}  \ar@{->}[r]&  \mathcal{F} |_{X_1}\ar@{->}[r]
 &(f^{*}\mathcal{G})^{\vee\vee}|_{X_1}  \\}
\end{equation*}
on $X_1$, which completes the proof from ${ \rm codim}(X \setminus X_1) \ge 2$.
\end{proof}

\begin{defn}[{\cite[Definition 2.3]{Druel}, \cite[Section 2.3]{LPT}}]
\label{alge_trans}
Let $X$ be a smooth projective variety and $\mathcal{F}$ be a foliation on $X$.
Then there exist a dominant rational map $f : X \dashrightarrow Y$  with connected fibers to a normal variety $Y$ and 
a foliation $\mathcal{G}$ on $Y$, such that  
$\mathcal{F} = f^{-1}\mathcal{G}$ and 
$\mathcal{G}$ is purely transcendental, i.e., there is no positive dimensional algebraic subvariety through a general point of $Y$ which is tangent to $\mathcal{G}$.

$\mathcal{G}$ is called the \textit{transcendental part} of $\mathcal{F}$ and 
the algebraically integrable foliation induced by $f$ is called the \textit{algebraic part} of $\mathcal{F}$.
\end{defn}

We give the construction of algebraic and transcendental parts for reader's convenience.
This construction is an application of \cite[Section 2.3]{LPT}.

If $\mathcal{F}$ is purely transcendental, we can take $\mathcal{G} =\mathcal{F}$ and $f = id$. 
We may assume that $\mathcal{F}$ is not purely transcendental.
According to \cite[Section 2.3]{LPT}, 
if there exists an $(n-k)$-dimensional algebraic subvariety tangent to $\mathcal{F}$ through a general point of $X$ for some nonnegative integer $k < n$, then the leaves of $\mathcal{F}$ are covered by $q$-dimensional algebraic subvarieties for some  integer $q \ge n-k$,
that is, 
there exist normal projective varieties $Y$ and $Z$, 
a dominant morphism $\pi_{1}: Z \rightarrow X$, 
and an equidimensional dominant morphism $\pi_{2} : Z \rightarrow Y$ with connected fibers, 
such that a fiber of $\pi_2$ has $q$ dimension and  $\pi_{1}(\pi_{2}^{-1}(y))$ is tangent to $\mathcal{F}$ for a general point $y \in Y$.

\begin{equation*}
\xymatrix@C=25pt@R=20pt{
Z \ar@{->}[d]_{\pi_1}  \ar@{->}[r]^{\pi_{2}} &Y \\
X&  \\   
}
\end{equation*}

Set
$$\widetilde{q}:= \max \Bigl \{0 \le q \le n: \txt{the leaves of $\mathcal{F}$ are covered by \\ $q$-dimensional algebraic subvarieties} \Bigr\}.$$
Then $\widetilde{q} > 0$ and we obtain the covering family $Y,Z,\pi_1,\pi_2$ as above such that a fiber of $\pi_2$ has $\widetilde{q}$ dimension.
By the maximality of $\widetilde{q}$ and the argument of \cite[Lemma 2.4]{LPT}, $\pi_{1}$ is birational. 
Hence $f : = \pi_{2} \circ \pi_{1}^{-1} : X \dashrightarrow Y$ is an essentially equidimensional dominant rational map with connected fibers.
Let $\mathcal{H}$ be an algebraically integrable foliation induced by $f$.
By Lemma \ref{rationaldecent}, there exists a foliation $\mathcal{G}$ on $Y$ with
$\mathcal{F} = f^{-1}\mathcal{G}$.

We show that $\mathcal{G}$ is purely transcendental.
To obtain a contradiction, suppose that there exists 
a $(\dim Y-k)$-dimensional algebraic subvariety tangent to $\mathcal{G}$ through a general point of $Y$ for some nonnegative integer $k < \dim Y$.
Hence the leaves of $\mathcal{G}$ are covered by $l$-dimensional algebraic subvarieties for some $l \in \mathbb{N}_{>0}$.
Therefore the leaves of $\mathcal{F}$ are covered by $(l + \widetilde{q})$-dimensional algebraic subvarieties, which contradicts the maximality of $\widetilde{q}$. Hence $\mathcal{G}$ is purely transcendental.

We have the  following lemma on an algebraic part of a foliation.
\begin{lem}\cite[Lemma 6.2]{Druel}
\label{almost_holomorphic}
Let $X$ be a smooth projective variety and $\mathcal{F}$ be a foliation on $X$.
If $\mathcal{F}$ is regular or $\mathcal{F}$ has a compact leaf,  
then the algebraic part of $\mathcal{F}$ has a compact leaf.
\end{lem}

\subsection{Slopes of torsion-free coherent sheaves}
\text{}

In this section, let $X$ be a smooth projective variety.
A {\it 1-cycle} is a formal linear combination of irreducible reduced proper curves.
$N_{1}(X)_{\mathbb{R}}$ is an $\mathbb{R}$-vector space of 1-cycles with real coefficients modulo numerical equivalence.
A class $\alpha \in N_{1}(X)_{\mathbb{R}}$ is {\it movable} if $D\alpha \ge 0$ for any effective Cartier divisor $D$ on $X$. The set of movable classes forms a closed convex cone $Mov(X) \subset N_{1}(X)_{\mathbb{R}}$, called the \textit{movable cone}.
 
Let $\mathcal{F}$ be a nonzero torsion-free coherent sheaf of $X$.
For any $\alpha  \in Mov(X) $, the {\it slope of $\mathcal{F}$ with respect to $\alpha$} is defined by 
$$
\mu_{\alpha }(\mathcal{F}) := \frac{ c_1(\mathcal{F})\alpha}{{\rm rk} \mathcal{F}}.
$$
$\mu_{\alpha }^{max}(\mathcal{F})$ is defined by the supremum of $\mu_{\alpha }(\mathcal{E})$ for any nonzero coherent subsheaf $\mathcal{E} \subset \mathcal{F}$ and 
$\mu_{\alpha }^{min}(\mathcal{F})$ is defined by the infimum of $\mu_{\alpha }(\mathcal{Q})$ for any nonzero torsion-free quotient sheaf $\mathcal{F} \rightarrow \mathcal{Q}$.
$\mathcal{F} $ is called $\alpha$-{\it semistable} if $\mu_{\alpha }^{max}(\mathcal{F})=\mu_{\alpha }(\mathcal{F})$.
By \cite[Proposition 2.4]{CP15}, 
there exists a unique subsheaf $\mathcal{F}_{max} \subset \mathcal{F}$, maximal with respect to the inclusion, such that 
$\mu_{\alpha }(\mathcal{F}_{max}) = \mu_{\alpha }^{max}(\mathcal{F}) $.
$\mathcal{F}_{max}$ is called the {\it maximal destabilizing subsheaf with respect to $\alpha$} of $\mathcal{F}$.
$\mathcal{F}_{max}$ is saturated and $\alpha$-semistable.

The following theorem states a relationship between foliations and slopes.
\begin{thm}\cite[Theorem 1.1]{CP15}
\label{CP15}
Let $\mathcal{F}$ be a foliation on $X$.
If there exists $\alpha \in Mov(X)$ such that $\mu_{\alpha }^{min}(\mathcal{F}) >0$, then $\mathcal{F}$ is an algebraically integrable foliation with rationally connected leaves, i.e., a Zariski closure of any leaf through a general point  is rationally connected.
\end{thm}

The following corollary was already proved in \cite[Corollarie 1.25]{Clauden}.
We give a proof for the reader's convenience.
\begin{cor}\cite[Corollarie 1.25]{Clauden}
\label{psef_transcendental}
Let $\mathcal{F}$ be a foliation on $X$.
If $\mathcal{F}$ is purely transcendental, then $K_{\mathcal{F}}$ is pseudo-effective.
\end{cor}
\begin{proof}
Assume that $K_{\mathcal{F}}$ is not pseudo-effective.
Then there exists  $\alpha \in Mov(X)$ such that 
$\mu_{\alpha }^{max}(\mathcal{F}) >0$.
We take a maximal destabilizing sheaf $\mathcal{E}$ with respect to $\alpha$ of $\mathcal{F}$. From 
$\mu_{\alpha}(\mathcal{E}) =\mu_{\alpha }^{min}(\mathcal{E}) =\mu_{\alpha }^{max}(\mathcal{F}) >0 $, 
by Theorem \ref{CP15}, $\mathcal{E}$ is an algebraically integrable foliation with rationally connected leaves, which is impossible.
\end{proof}

\subsection{Algebraic positivities of torsion-free coherent sheaves}
\text{}

In this section, we recall notions of singular hermitian metrics and algebraic positivities of torsion-free coherent sheaves. Throughout this paper, we adopt the definition of singular hermitian metrics in \cite{HPS18}.

\begin{defn}
Let $X$ be a smooth $n$-dimensional projective variety.
 \begin{enumerate}
  \item \cite[Definition 1.9]{DPS} A vector bundle $E$ is {\it ample} (resp. {\it strctly nef, nef}) if $\mathcal{O}_{\mathbb{P}(E) }(1)$ is ample (resp. strctly nef, nef) on $\mathbb{P}(E) $.
 \item  \cite[Definition 1.17]{DPS} A vector bundle $E$ is {\it numerically flat} if $E$ is nef and $c_{1}(E) = 0$.
 \item \cite[Definition 3.20]{Nak} A torsion-free coherent sheaf $\mathcal{E}$ is {\it weakly positive at $x \in X$} if
 for any $a \in \mathbb{N}_{>0}$ and for any ample line bundle $A$ on $X$, there exists 
$b \in \mathbb{N}_{>0}$ such that $\Sym^{ab}( \mathcal{E}  ) ^{\vee\vee} \otimes A^{ b}$ is globally generated at $x$. 
 \item \cite[Definition 3.20]{Nak} A torsion-free coherent sheaf $\mathcal{E}$ is {\it pseudo-effective} ({\it weakly positive in the sense of Nakayama}) if  $ \mathcal{E}  $ is weakly positive at some $x \in X$.
 \item \cite[Definition 3.20]{Nak} A torsion-free coherent sheaf $\mathcal{E}$ is {\it V-big} ({\it big in the sense of Viehweg})  if there exist $a \in \mathbb{N}_{>0}$ and 
 an ample line bundle $A$ on $X$ such that $\Sym^{a}( \mathcal{E}  ) ^{\vee\vee} \otimes A^{-1}$ is pseudo-effective.
   \item \cite[Definition 6.4]{DPS01} A torsion-free coherent sheaf $\mathcal{E}$ is {\it almost nef} if there exists a countable family  of proper subvarieties $Z_{i}$ of $X$ 
  such that $\mathcal{E} |_{C} $ is a nef vector bundle on $C$ for any curve $ C \not \subset \cup_{i} Z_i$.
 \item \cite[Definition 19.1]{HPS18} A torsion-free coherent sheaf $\mathcal{E}$ \textit{has a positively curved singular hermitian metric} if $\mathcal{E}$ has a singular hermitian metric $h$ such that 
$\log |u|_{h^{\vee}}$ is a psh function 
for any local section $u$ of $\mathcal{E}^{\vee}$, 
where $h^{\vee}$ is the dual metric defined by $h^{\vee}  := {}^t\! h^{-1}$. 
\item \cite[Definition 2.1]{Pet1} For any ample divisors $H_1, \ldots, H_{n-1}$ on $X$, a torsion-free coherent sheaf $\mathcal{E}$ is \textit{$(H_1, \ldots, H_{n-1})$-generically ample}
if $\mathcal{E}|_{C}$ is ample on $C$ for a general curve $C = D_{1} \cap \dots \cap D_{n-1}$ with general $D_{i} \in | m_i H_i|$ and $m_i \gg 0 $.
 \end{enumerate}
 \end{defn}
 
 \begin{rem}\label{rem-psef}
The notion of pseudo-effectivity is often used in a different meaning. 
For example, in other papers, we may say a vector bundle $E$ on a smooth projective variety $X$ is pseudo-effective 
when $\mathcal{O}_{\mathbb{P}(E) }(1)$ is  a pseudo-effective line bundle. 
However, the pseudo-effectivity of $E$ in this paper is stronger than this definition. 
We require that  the image of the non-nef locus of 
$\mathcal{O}_{\mathbb{P}(E) }(1)$ is properly contained in $X$  
in addition to this condition. 
 \end{rem}
 
Relationships among them can be summarized  by the following table: 

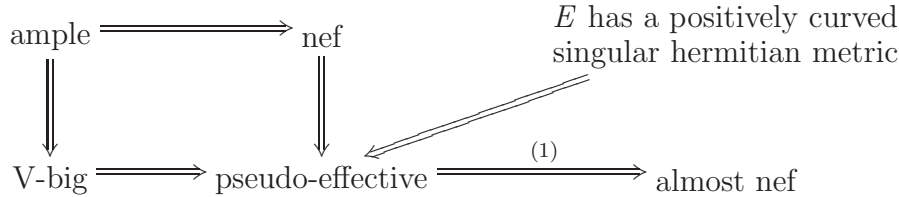
\begin{equation*}
\xymatrix@C=40pt@R=30pt{
  \txt{ample} \ar@{=>}[r]\ar@{=>}[d]& \txt{nef} \ar@{=>}[d] &
   \txt{\textit{E} has a positively curved \\ singular hermitian metric} \ar@{=>}[ld] \\ 
 \txt{V-big} \ar@{=>}[r] &  \txt{pseudo-effective} \ar@{=>}[r]^{\ \ \ \  (1) }& \txt{almost nef}
}
\end{equation*}

Notice that when $E$ is a line bundle, the converse of (1) holds by \cite[Theorem 0.2]{BDPP} and  $E$ is V-big if and only if $E$ is big. 

\begin{prop}
\label{MR}
Let $X$  be a smooth $n$-dimensional projective variety, $\mathcal{F}$ be a torsion-free coherent sheaf of $X$, and $H_1, \ldots, H_{n-1}$ be ample divisors on $X$.
Set $\alpha := H_1 \cdots H_{n-1}$.
If $\mu_{\alpha }^{min}(\mathcal{F}) >0$, then $\mathcal{F}$ is $(H_1, \ldots, H_{n-1})$-generically ample. 
\end{prop}

\begin{proof}
The proof is the same as the proof of \cite[Theorem 0.3]{Ou}.
By Mehta-Ramanathan's theorem in \cite[Theorem 6.1]{MR}, 
a restriction of the Harder-Narasimhan filtration with respect to $\alpha$ of $\mathcal{F}$ to $C$ is
the Harder-Narasimhan filtration of $\mathcal{F}|_{C}$, for a general curve $C = D_{1} \cap \dots \cap D_{n-1}$ with general $D_{i} \in | m_i H_i|$ and $m_i \gg 0 $.
Hence $\mu^{min}(\mathcal{F}|_{C}) = \mu_{\alpha }^{min}(\mathcal{F}) >0$, and finally $\mathcal{F}|_{C}$ is ample on $C$.
\end{proof}

The following theorem is often used, thus we will summarize it.
\begin{thm}
\label{almost_nef}

Let $X$ be a smooth projective variety and $\mathcal{E}$ be an almost nef torsion-free coherent sheaf of $X$.
Then the following statements hold.

\begin{enumerate}
\item \cite[Proposition 3.8]{LOY20} For any generically surjective morphism $\tau : \mathcal{E} \rightarrow \mathcal{Q} $ to a non-zero torsion-free coherent sheaf, $\mathcal{Q}$ is almost nef.
In particular, $\mu_{\alpha}^{min}(\mathcal{E}) \ge 0$ for any $\alpha \in Mov(X)$.
\item \cite[Theorem 6.7]{DPS01} If $\mathcal{E}$ is a vector bundle, for any non zero section $s \in H^0(X,\mathcal{E}^{\vee})$, 
$s$ has no zero point.
\item \cite[Theorem 1.8]{HP19} \cite[Theorem 4.4]{LOY20} If $c_1(\mathcal{E}) = 0$ and $\mathcal{E}$ is reflexive, then $\mathcal{E}$ is a numerically flat vector bundle.
\end{enumerate}

\end{thm}

\section{Fujita's decomposition}

\begin{lem}
\label{Fujitafiltration}
Let $X$ be a smooth $n$-dimensional projective variety,  $\mathcal{F}$ be a reflexive coherent sheaf of $X$, and $H_1, \ldots, H_{n-1}$ be ample divisors on $X$.
Set $\alpha := H_1 \cdots H_{n-1} \in Mov(X)$. 
If $\mathcal{F}$ is almost nef, then there exist, up to isomorphism, unique torsion-free coherent sheaves $\mathcal{Q},\mathcal{K}$ which satisfy the following conditions.
\begin{enumerate}
\item $\mathcal{Q}$ is a numerically flat vector bundle. 
\item $\mu_{\alpha}^{max}(\mathcal{K})<0$ unless $\mathcal{K} = 0$.
\item We have the following exact sequence 
\begin{equation*}
\xymatrix@C=25pt@R=20pt{
0\ar@{->}[r]& \mathcal{Q}  \ar@{->}[r]&  \mathcal{F}^{\vee} \ar@{->}[r]
 &\mathcal{K}  \ar@{->}[r]&0. \\}
\end{equation*}
\end{enumerate}
Moreover, $\mathcal{Q},\mathcal{K}$ are independent of the choice of ample divisors $H_1, \ldots, H_{n-1}$. 
\end{lem}

\begin{proof}
(Existence.)
If $\mu_{\alpha}^{min}(\mathcal{F})>0$, we can take $\mathcal{Q} =0$ and $\mathcal{K}  = \mathcal{F}^{\vee}$.
From $\mu_{\alpha}^{min}(\mathcal{F})\ge0$, 
we may assume that $\mu_{\alpha}^{min}(\mathcal{F})=0$.
Let $\mathcal{Q} $ be the maximal destabilizing sheaf with respect to $\alpha$ of $\mathcal{F}^{\vee}$. Set $\mathcal{K} := \mathcal{F}^{\vee} / \mathcal{Q}$.
Condition (3) is clear, thus we prove conditions (1) and (2).

By Theorem \ref{almost_nef}, $\mathcal{Q}^{\vee}$ is almost nef.
Since $-c_1(\mathcal{Q})$ is pseudo-effective and $\mu_{\alpha} (\mathcal{Q}) =\mu_{\alpha}^{max}(\mathcal{F}^{\vee})=0$, 
we have $c_1(\mathcal{Q}) =0$, and finally $\mathcal{Q}$ is a numerically flat vector bundle by Theorem \ref{almost_nef}.

We show the condition (2).
To obtain a contradiction, suppose that $\mu_{\alpha}^{max}(\mathcal{K})\ge0$ and $\mathcal{K}\neq 0$.
We take a maximal destabilizing sheaf $\mathcal{K}' $ with respect to $\alpha$ of $\mathcal{K}$. 
Then there exists $\mathcal{Q}' \subset \mathcal{F}^{\vee}$ such that 
\begin{equation*}
\xymatrix@C=25pt@R=20pt{
0\ar@{->}[r]& \mathcal{Q}  \ar@{->}[r]&  \mathcal{Q}' \ar@{->}[r]
 &\mathcal{K}'  \ar@{->}[r]&0. \\}
\end{equation*}
Hence 
$
0 \ge {\rm rk}\mathcal{Q}' \mu_{\alpha} (\mathcal{Q}') ={\rm rk}\mathcal{K}'  \mu_{\alpha} (\mathcal{K}' ) \ge 0, 
$
which contradicts the maximality of $\mathcal{Q}$.

(Uniquness.)
We assume that there exist torsion-free coherent sheaves $\mathcal{Q},\mathcal{K}$ and $\mathcal{Q}',\mathcal{K}'$ which satisfy the conditions (1)-(3).
Then we get the following exact sequence
\begin{equation*}
\xymatrix@C=25pt@R=20pt{
0\ar@{->}[r]&\mathcal{Q}  \ar@{->}[r]&\mathcal{F}^{\vee} \ar@{->}[r]\ar@{=}[d]&\mathcal{K}\ar@{->}[r]&0 \\
0\ar@{->}[r]&\mathcal{Q}'  \ar@{->}[r]&\mathcal{F}^{\vee} \ar@{->}[r]&\mathcal{K}'\ar@{->}[r]&0. \\   
}
\end{equation*}
We thus get a map $\beta: \mathcal{Q} \rightarrow \mathcal{K}'$. 
From $\mu_{\alpha}^{min}(\mathcal{Q}) =0>\mu_{\alpha}^{max}(\mathcal{K}')$,
$\beta$ is a zero map. 
Hence we obtain an injective morphism $ \mathcal{Q} \rightarrow \mathcal{Q}' $.
By a similar argument, we obtain an injective morphism $ \mathcal{Q}' \rightarrow \mathcal{Q}$, therefore 
$ \mathcal{Q}'\cong\mathcal{Q}$.

The proof of the last statement is the same as that of uniqueness,
since any numerically flat vector bundle is $\alpha'$-semistable, where $\alpha' := H'_1\cdots H'_{n-1}$ for any ample divisors $H'_1, \ldots, H'_{n-1}$ on $X$.
\end{proof}

\begin{thmdefn}
\label{Fujita_positive}
Under the hypotheses of Lemma \ref{Fujitafiltration}, we assume that $\mathcal{F}$ is an almost nef vector bundle.
Then there exist unique vector bundles $Q,G$ which satisfy the following conditions.
\begin{enumerate}
\renewcommand{\theenumi}{\alph{enumi}}
\item $Q$ is numerically flat. 
\item $G$ is almost nef and  $\mu_{\alpha}^{min}(G)>0$  unless $G = 0$.
\item We have the following exact sequence of vector bundles
\begin{equation*}
\xymatrix@C=25pt@R=20pt{
0\ar@{->}[r]& G \ar@{->}[r]& \mathcal{F}  \ar@{->}[r]
 &Q \ar@{->}[r]&0. \\}
\end{equation*}
\end{enumerate}
$G$ is called the {\it positive part} of $\mathcal{F}$.
\end{thmdefn}
\begin{proof}
There exist unique torsion-free coherent sheaves $\mathcal{Q},\mathcal{K}$ which satisfy the conditions (1)-(3) in Lemma \ref{Fujitafiltration}. We have the following exact sequence 
\begin{equation*}
\xymatrix@C=25pt@R=20pt{
0\ar@{->}[r]& \mathcal{Q} \ar@{->}[r]^{\tau \,\,\,\,\,\,}& \mathcal{F}^{\vee}  \ar@{->}[r]
 &\mathcal{K} \ar@{->}[r]&0. \\}
\end{equation*}
Set $k:={\rm rk}\mathcal{Q}$. 
Since 
$\wedge^{k}\mathcal{F}\otimes \det \mathcal{Q}$ is almost nef, 
$\wedge^{k}\tau $ has no zero point by Theorem \ref{almost_nef}.
Hence $\tau$ is an injective morphism between vector bundles and $\mathcal{K}$ is a vector bundle. 
Set $G := \mathcal{K}^{\vee}$ and $Q:=\mathcal{Q}^{\vee}$, 
the proof is complete. 
\end{proof}

\begin{ex}
A positive part is not necessarily a maximal destabilizing sheaf. 
Set $E := \mathcal{O}_{\mathbb{CP}^{1}}(1)\oplus \mathcal{O}_{\mathbb{CP}^{1}}(2)$, then 
$E$ is a nef vector bundle on $\mathbb{CP}^1$.
The maximal destabilizing sheaf of $E$ is $\mathcal{O}_{\mathbb{CP}^{1}}(2)$.
However, the positive part of $E$ is $E$.

\end{ex}

\begin{thmdefn}
\label{fujitadecom}
Under the hypotheses of Lemma \ref{Fujitafiltration}, we assume that $\mathcal{F}$ has a positively curved singular hermitian metric.
Then there exist unique reflexive coherent sheaves $Q, G$ which satisfy the following conditions.
\begin{enumerate}
\renewcommand{\theenumi}{\alph{enumi}}
\item $Q$ is a hermitian flat vector bundle.
\item$G$ has a positively curved singular hermitian metric and $\mu_{\alpha}^{min}(G)>0$ unless $G =0$.
\item $\mathcal{F} \cong Q \oplus G$.
\end{enumerate}
This decomposition is called {\it Fujita's decomposition of $\mathcal{F}$}.
\end{thmdefn}
\begin{proof}
There exist unique torsion-free coherent sheaves $\mathcal{Q},\mathcal{K}$ which satisfy the conditions (1)-(3) in Lemma \ref{Fujitafiltration}.
Let $X_{0}$ denote the maximal Zariski open set where $\mathcal{F}$ is locally free.
Since $\mathcal{F}$ has a positively curved singular hermitian metric,
$$
 \mathcal{F}^{\vee}|_{X_{0}} \cong \mathcal{Q}|_{X_{0}} \oplus \mathcal{K}|_{X_{0}}
$$
by \cite[Theorem 1.3]{HIM}.
Write $G := \mathcal{K}^{\vee}$ and $Q:=\mathcal{Q}^{\vee}$, then 
$\mathcal{F} \cong Q \oplus G.$
Since $c_{1}(Q)=0$ and $Q$ has a positively curved singular hermitian metric, $Q$ is a hermitian flat vector bundle by \cite[Lemma 3.5]{HIM}.
\end{proof}

In particular, we obtain Fujita's decomposition theorem on a direct image sheaf of a relative pluricanonical line bundle.
\begin{cor}[{\textit{cf.} \cite[Theorem]{Fujita}, \cite[Theorem 1.1]{CD}, \cite[Theorem 2]{CK}}]
 \label{Fujita_decomposition}
Let $f : X \rightarrow Y$ be a surjective morphism with connected fibers from a compact K\"ahler manifold $X$ to a smooth $d$-dimensional projective variety $Y$.
For any $m \in \mathbb{N}_{>0}$, 
there is a unique decomposition $$( f_{*}(mK_{X/Y}) )^{\vee\vee} \cong Q \oplus G,$$
where $Q$ is a hermitian flat vector bundle and $G$ is an $(H_1, \ldots, H_{d -1})$-generically ample reflexive coherent sheaf for any ample divisors $H_1, \ldots, H_{d -1}$ on $Y$.

In particular, if $Y$ is a curve, 
for any $m \in \mathbb{N}_{>0}$, 
there is a unique decomposition $$ f_{*}(mK_{X/Y})  \cong Q \oplus G,$$
where $Q$ is a hermitian flat vector bundle and $G$ is an ample vector bundle.

\end{cor}
\begin{proof}
By \cite[Theorem B]{wang}, 
$f_{*}(mK_{X/Y})$ has a positively curved singular hermitian metric, which completes the proof by Proposition \ref{MR} and Theorem-Definition \ref{fujitadecom}.
\end{proof}

\section{Almost nef regular foliations}
\subsection{Structure theorems of almost nef regular foliations}

\begin{thm}
\label{psefregular}
Let $X$ be a smooth $n$-dimensional projective variety and 
$\mathcal{F}$ be a rank $r$ almost nef regular foliation on $X$ with $c_1(\mathcal{F})\neq0$.
Then there exist a smooth morphism $f : X \rightarrow Y$ between smooth projective varieties and a numerically flat regular foliation $\mathcal{G}$ on $Y$, such that all fibers of $f$ are rationally connected,
$\mathcal{F}=f^{-1}\mathcal{G}$, and there is an exact sequence
\begin{equation*}
\xymatrix@C=25pt@R=20pt{
0\ar@{->}[r]&T_{X/Y} \ar@{->}[r]&\mathcal{F} \ar@{->}[r]
 & f^{*}\mathcal{G} \ar@{->}[r]&0. \\}
\end{equation*}
 
\end{thm}
\begin{proof}
Let $H_1, \ldots, H_{n-1}$ be ample divisors on $X$ and $\alpha := H_1 \cdots H_{n-1}$. 
We take the positive part $G$ of $\mathcal{F}$ by Theorem-Definition \ref{Fujita_positive}.  
Then $G$ is a subbundle of $\mathcal{F}$, $\mu^{min}_{\alpha}(G) >0$, and 
$\mu^{max}_{\alpha}(\mathcal{F}/G) = \mu_{\alpha}(\mathcal{F}/G)=0$.
By Theorem \ref{CP15},
$G$ is an algebraically integrable regular foliation with rationally connected leaves. 
By \cite[Corollary 2.11]{Hor07}, we obtain 
a smooth morphism $f : X \rightarrow Y$ between smooth projective varieties such that all fibers of $f$ are rationally connected and $G =T_{X/Y}$.
By Lemma \ref{decent}, there exists a regular foliation $\mathcal{G}$ on $Y$ such that
$\mathcal{F}=f^{-1}\mathcal{G}$ and 
\begin{equation*}
\xymatrix@C=25pt@R=20pt{
0\ar@{->}[r]&T_{X/Y} \ar@{->}[r]&\mathcal{F} \ar@{->}[r]
 & f^{*}\mathcal{G} \ar@{->}[r]&0. \\}
\end{equation*}
Since $f^{*}\mathcal{G}$ is numerically flat, 
$\mathcal{G}$ is so.
\end{proof}

\begin{cor}
\label{ampleregular}
Under the hypotheses of Theorem \ref{psefregular}, let $F$ be a fiber of $f$.  
Then the following statements hold.
\begin{enumerate}
\item If $\mathcal{F}$ is ample, then $X$ is isomorphic to $\mathbb{CP}^n$.
\item If $\mathcal{F}$ is nef, then $F$ is Fano.
\item If $\mathcal{F}$ is V-big, then $\mathcal{F} = T_{X/Y}$ and $F$ is isomorphic to $\mathbb{CP}^{r}$.
\item If $\mathcal{F}$ is nef and V-big, then $X$ is isomorphic to $\mathbb{CP}^n$.
\end{enumerate}
 
\end{cor}

\begin{proof}
Since  $f^{*}\mathcal{G}$ is a numerically flat, 
if $ \mathcal{F}$ is ample (resp. nef, V-big),
then both $T_{X/Y}$ and $T_{F}$ are so
by Theorem \ref{psefregular}.

(1).
Since both $T_{X/Y}$ and $-K_{X/Y}$ are ample, $\dim Y =0$ and $X \cong \mathbb{CP}^n$ by \cite[Corollary 2.8]{KMM} and \cite[Theorem 8]{Mori}.


(2).
$F$ is rationally connected and $T_F$ is nef.
Hence $F$ is Fano by \cite[Proposition 3.10]{DPS}.

(3).
Since both $f^{*}\mathcal{G}$ and $T_{F}$ are V-big, it follows by \cite[Corollary 7.8]{Mura} and \cite[Corollary 1.3]{Iwai}.

(4). The proof is the same as that of (1) by \cite[Theorem 1.1]{Druel}, \cite[Theorem 1.6]{Ejiri}, \cite[Corollary 7.8]{Mura}, and \cite[Corollary 1.3]{Iwai}.
\end{proof}

\begin{rem}
\label{strictl_nef}
By our method in Corollary \ref{ampleregular}, we only know that, if $T_X$ contains a strictly nef regular foliation, then there exists a smooth morphism $f : X \rightarrow Y$ between smooth projective varieties such that all fibers of $f$ are isomorphic to $\mathbb{CP}^{d}$ for some $d \in \mathbb{N}_{>0}$.
However, in \cite[Theorem 1.3]{LOY20}, it was proved that, 
if $T_X$ contains a strictly nef locally free sheaf, then $X$ admits a $\mathbb{CP}^{d}$-bundle structure $f : X \rightarrow Y$ and $Y$ is Brody hyperbolic, i.e., every holomorphic map $h : \mathbb{C} \rightarrow Y$ is constant.
\end{rem}

\begin{ex}

Let $C$ be a smooth curve, $e$ be a nonnegative integer, and $p$ be a point of $C$. 
Set $E:= \mathcal{O}_{C} \oplus \mathcal{O}_{C}(-ep)$ and $X := \mathbb{P}(E)$.
Let $\pi: X \rightarrow C$ be the ruling 
and $\mathcal{F}$ be a foliation induced by $\pi$.
By Corollary \ref{psef_ruled_surface}, $\mathcal{F}$ is nef if and only if $e=0$, and 
$\mathcal{F}$ is V-big if and only if $e>0$.
Therefore $\mathcal{F}$ can be nef or V-big regardless of the genus of $C$.
Notice that $\mathcal{F}$ cannot be strictly nef in this example.

\end{ex}

\subsection{Positive parts and algebraic parts of almost nef regular foliations}
\text{}

In this section, let $X$ be a smooth $n$-dimensional projective variety,  $\mathcal{F}$ be an almost nef regular foliation on $X$ with $c_1(\mathcal{F})\neq0$, and
$H_1, \ldots, H_{n-1}$ be ample divisors on $X$.
Set $\alpha := H_1 \cdots H_{n-1}$.
We study a relationship between a positive part of $\mathcal{F}$ and an algebraic part of $\mathcal{F}$.

\begin{lem}
\label{psefalgebraic}
An algebraic part $\mathcal{H}$ of $\mathcal{F}$ is an almost nef regular foliation with $K_{\mathcal{F}} \equiv K_{\mathcal{H}}$.
\end{lem}
\begin{proof}

Let $f : X \dashrightarrow Y$ be a dominant rational map induced by $\mathcal{H}$, and let $\mathcal{G}$ be a transcendental part of $\mathcal{F}$.
Since $\mathcal{F}$ is almost nef, 
$\det (\mathcal{F}/\mathcal{H})$ is pseudo-effective, hence
$(K_{\mathcal{H}} - K_{\mathcal{F}} )\alpha \ge 0$.
By the proof of \cite[Proposition 6.1]{Druel}, $K_{\mathcal{H}} - K_{\mathcal{F}} = -(f^{*}K_{\mathcal{G}} + R)$ for some effective divisor $R$ on $X$, hence $(K_{\mathcal{H}} - K_{\mathcal{F}} )\alpha \le 0$ by Corollary \ref{psef_transcendental} or \cite[Proposition 6.3]{Druel}.
Thus we obtain $K_{\mathcal{F}} \equiv K_{\mathcal{H}}$.

Set $\mathcal{Q} = \mathcal{F}/\mathcal{H}$, then there exists a coherent sheaf $\mathcal{K}$ such that
 \begin{equation*}
\xymatrix@C=25pt@R=20pt{
0\ar@{->}[r]&\mathcal{Q}^{\vee} \ar@{->}[r]^{\tau }
& \mathcal{F}^{\vee}  \ar@{->}[r] &\mathcal{K} \ar@{->}[r]&0. \\}
\end{equation*}
By Theorem \ref{almost_nef}, $\mathcal{Q}^{\vee}$ is numerically flat, $\mathcal{K}$ is a vector bundle,
and the above exact sequence is an exact sequence of vector bundles.
Since $\mathcal{H}$ is equal to $\mathcal{K}^{\vee}$ outside a codimension two Zariski closed subset, $\mathcal{H} = \mathcal{K}^{\vee}$ and $\mathcal{H} $ is an almost nef regular foliation.
\end{proof}

Let $\mathcal{H}$ be an algebraic part of $\mathcal{F}$ and
$f : X \dashrightarrow Y$ be a dominant rational map induced by $\mathcal{H}$.
We take a birational morphism $\pi: \widetilde{X} \rightarrow X$ and a surjective morphism $\widetilde{f} : \widetilde{X} \rightarrow Y$ with connected fibers such that $\tilde{f} = f \circ \pi$.
Now we consider a relative MRC fibration of $\widetilde{f} $.
By Lemma \ref{raynord_foliation}, 
we may assume that a relative MRC fibration $\widetilde{r} : \widetilde{X} \rightarrow \widetilde{Z}$ of $\widetilde{f} $ is a morphism, 
both $\widetilde{f}$-excptional divisors and  $\widetilde{r}$-excptional divisors are $\pi$-exceptional, 
and $\widetilde{X}, Y, \widetilde{Z} $ are smooth projective varieties. 
Set an induced morphism $\varphi : \widetilde{Z} \rightarrow Y$.

\begin{equation*}
\xymatrix@C=25pt@R=20pt{
\widetilde{X} \ar@{->}[d]_{\pi}  \ar@{->}[rd]^{\widetilde{f}} \ar@{->}[r]^{\widetilde{r}} & \widetilde{Z} \ar@{->}[d]^{\varphi} \\
X\ar@{-->}[r]_{f} & Y \\   
}
\end{equation*}

Let $\mathcal{R}$ denote the algebraically integrable foliation induced by a dominant  rational map $r:= \widetilde{r} \circ \pi^{-1} : X \dashrightarrow \widetilde{Z}$.

\begin{thm}
\label{posiotive_mrc}
Under the hypotheses stated above, $\mathcal{R}$ is a positive part of $\mathcal{F}$.
In particular, $\mathcal{R}$ is an almost nef regular algebraically  integrable foliation.
\end{thm}

In other words, if $\mathcal{F}$ is an almost nef regular foliation, then a positive part of $\mathcal{F}$ is a foliation induced by a relative MRC fibration of an algebraic part of $\mathcal{F}$.
In particular, the smooth morphism $f$ in Theorem \ref{psefregular} is a relative MRC fibration of an algebraic part of $\mathcal{F}$.

\begin{proof} The proof is divided into 2 steps.
\begin{step}
We show that $\mathcal{R}$ is an almost nef regular foliation and
$K_{\mathcal{R}} \equiv K_{\mathcal{H}} \equiv K_{\mathcal{F}} $.

By Proposition \ref{raynord_foliation}, there exist $\pi$-exceptional divisors $E_{\mathcal{H}},E_{\mathcal{R}}$ such that 
$$
\pi^{*}K_{\mathcal{H}} \sim_{\mathbb{Q}} K_{\widetilde{X}/Y} -Ram(\widetilde{f}) + E_{\mathcal{H}}
\,\,\,\,\, \text{and} \,\,\,\,\,
\pi^{*}K_{\mathcal{R}} \sim_{\mathbb{Q}} K_{\widetilde{X}/\widetilde{Z}} -Ram(\widetilde{r}) + E_{\mathcal{R}}.
$$
Thus we obtain
$$
\pi^{*}(K_{\mathcal{H}}-K_{\mathcal{R}} )\sim_{\mathbb{Q}} 
\widetilde{r}^{*}( K_{\widetilde{Z}/Y} - Ram(\varphi))
+( Ram(\widetilde{r})+\widetilde{r}^{*}Ram(\varphi) -Ram(\widetilde{f})) + (E_{\mathcal{H}} - E_{\mathcal{R}}).
$$
Since $\widetilde{r}$ is a relative MRC fibration, 
$K_{\widetilde{Z}/Y} - Ram(\varphi)$ is pseudo-effective
by the argument of \cite[Theoreme 1.20]{Clauden}.
Since both $\widetilde{f}$-excptional divisors and  $\widetilde{r}$-excptional divisors are $\pi$-exceptional, 
$$( Ram(\widetilde{r})+\widetilde{r}^{*}Ram(\varphi) -Ram(\widetilde{f}))\alpha \ge 0$$ by the argument of Lemma \ref{ramification}.
Hence $(K_{\mathcal{H}}-K_{\mathcal{R}})\alpha \ge 0$.
Since $ \mathcal{R} \subset \mathcal{H}$ and $\mathcal{H}$ is almost nef, 
$(K_{\mathcal{H}}-K_{\mathcal{R}})\alpha \le0$.
Therefore $K_{\mathcal{R}} \equiv K_{\mathcal{H}} \equiv K_{\mathcal{F}} $.
By the argument of Lemma \ref{psefalgebraic}, 
$\mathcal{R}$ is an almost nef regular foliation and $\mathcal{F}/\mathcal{R}$ is a numerically flat vector bundle.
\end{step}

\begin{step}

We show that $\mu_{\alpha}^{min}(\mathcal{R})>0$.

Since $\mathcal{R}$ is an almost nef regular foliation with rationally connected leaves, there exists a smooth morphism $g : X \rightarrow Z$ with $\mathcal{R} =T_{X/Z}$ by \cite[Corollary 2.11]{Hor07}.
To obtain a contradiction, suppose that $\mu_{\alpha}^{min}(\mathcal{R})=0$.

By Theorem-Definition \ref{Fujita_positive},
there exists an almost nef regular foliation $\mathcal{R}' \subset \mathcal{R}$ with  $\mu_{\alpha}^{min}(\mathcal{R}')>0$ and $K_{\mathcal{R}} \equiv K_{\mathcal{R}'} $.
By Theorem \ref{CP15} and \cite[Corollary 2.11]{Hor07}, 
we obtain a smooth morphism $g' : X \rightarrow Z'$ with rationally connected fibers with $\mathcal{R}' =T_{X/Z'}$.
Notice that $\mathcal{R}/\mathcal{R}'$ is numerically flat vector bundle and $\dim Z < \dim Z'$ from $\mathcal{R}' \subsetneq \mathcal{R}$.

\begin{equation*}
\xymatrix@C=25pt@R=20pt{
X \ar@{->}[rd]^{g} \ar@{->}[r]^{g'} \ar@{-->}[d]_{f} 
& Z' \\
Y& Z \\   
}
\end{equation*}

We show that there exists a smooth morphism $h : Z' \rightarrow Z$ such that $h \circ g' =g$.
From $\mathcal{R}' \subset \mathcal{R}$, any fiber of $g'$ is contained in a fiber of $g$.
Set 
$$\Gamma :=\{ (g(x) , g'(x)) \in Z \times Z'  : x \in X\}.$$
Since $g,g'$ are smooth, $\Gamma$ is a $\dim Z'$-dimensional submanifold of $ Z \times Z' $.
By \cite[Lemma 1.15]{Debbare}, $pr_{2}|_{\Gamma}: \Gamma \rightarrow Z'$ is isomorphism. Hence  $h := pr_1 \circ (pr_{2}|_{\Gamma})^{-1}$ is smooth. 

Therefore we obtain the following exact sequence
\begin{equation*}
\xymatrix@C=25pt@R=20pt{
0 \ar@{->}[r] & g'^{*} \Omega_{Z'/Z} \ar@{->}[r] &\mathcal{R}^{\vee}= \Omega_{X/Z}  \ar@{->}[r] & \mathcal{R}'^{\vee}=\Omega_{X/Z'} 
 \ar@{->}[r]&0. \\}
\end{equation*}
Set $X_z := g^{-1}(z)$ for a general point $z \in Z$.
Then we have 
$g'^{*} \Omega_{Z'/Z}|_{X_{z}} \subset \Omega_{X/Z}|_{X_{z}}$.
From $\dim Z < \dim Z'$, 
$ \Omega_{X_{z}}$ contains a non trivial numerically flat vector bundle, which is impossible since $X_z$ is rationally connected.
\end{step}
\end{proof}

\begin{cor}[{\textit{cf.} \cite[Theorem 1.1]{HIM}, \cite[Conjecture 4.12]{Pet1}}]
\label{psef_example}

Let $X$ be a smooth $n$-dimensional projective variety and 
$\mathcal{G}$ be a foliation on $X$.
Assume that $T_X$ is almost nef.
Then the followings are equivalent. 
\begin{enumerate}
\item $\mu_{\alpha}^{min}(\mathcal{G}) > 0$ and $K_{\mathcal{G}} \equiv K_X$, where $\alpha = H_1 \cdots H_{n-1}$ for some ample divisors $H_1, \ldots, H_{n-1}$ on $X$.
\item $\mathcal{G}$ is a positive part of $T_X$.
\item There exists a smooth MRC morphism $f:X \rightarrow Y$ such that $\mathcal{G} = T_{X/Y}$.
\end{enumerate}

Moreover, there exist a finite \'{e}tale morphism $\pi : \widetilde{X} \rightarrow X $ and a smooth surjective morphism $\widetilde{f} : \widetilde{X} \rightarrow A$ to an Abelian variety $A$ such that all fibers of $\widetilde{f}$ are rationally connected.

In particular, if $T_X$ is almost nef and $X$ is rationally connected, then $T_X$ is $(H_1, \ldots, H_{n-1})$-generically ample for any ample divisors $H_1, \ldots, H_{n-1}$ on $X$.

\end{cor}
\begin{proof}
We regard $T_X$ as a foliation. Then the equivalence of the three conditions is clear by Theorem \ref{posiotive_mrc}.
Assume that $\mathcal{G}$ is the positive part of $T_X$.
By condition (3), we have the following exact sequence
 \begin{equation*}
\xymatrix@C=25pt@R=20pt{
0\ar@{->}[r]&\mathcal{G} = T_{X/Y} \ar@{->}[r]&T_X\ar@{->}[r]&f^{*}T_Y\ar@{->}[r]&0. \\
}
\end{equation*}
Then $f^{*}T_Y$ is numerically flat, hence $T_Y$ is so, and finally there exists a finite \'{e}tale morphism $\sigma : A \rightarrow Y$ from an Abelian variety $A$.
Set $\widetilde{X}:= X \times_{Y} A$, then the second statement holds.

if $T_X$ is almost nef and $X$ is rationally connected, then $T_X$ is the positive part of it, hence $\mu_{\alpha}^{min}(T_X) >0$ for $\alpha := H_1 \cdots H_{n-1}$.
By Proposition \ref{MR}, $T_X$ is $(H_1, \ldots, H_{n-1})$-generically ample.
\end{proof}

\subsection{Classifications of almost nef regular foliations on surfaces}\label{Classification}
\text{}

In this section, let $S$ be a smooth projective surface and $\mathcal{F}$ be an almost nef regular foliation on $S$.
For the classifications, we need the some of standard facts on divisors of ruled surfaces.
We summarize it in Appendix \ref {ruled_surface}.

\subsubsection{Case of ${\rm rk} \mathcal{F} = 1$ and $c_1(\mathcal{F}) \neq 0$.}
\begin{prop}
A tangent bundle $T_S$ contains a rank 1 almost nef regular foliation $\mathcal{F}$ with $c_1(\mathcal{F}) \neq 0$
if and only if $S$ is a ruled surface over a smooth curve $C$ and $\mathcal{F}= - K_{S/C}$.

\end{prop}
\begin{proof}
It follows by Theorem \ref{psefregular} and Corollary \ref{psef_ruled_surface}.
\end{proof}

\subsubsection{Case of ${\rm rk} \mathcal{F} = 1$ and $c_1(\mathcal{F}) = 0$.}
\text{}

By \cite[Th\'eor\`eme 1.2]{Tou08} and \cite[Lemma 5.9]{Druel16I}, 
we have the following theorem. 
\begin{thm}[{\cite[Th\'eor\`eme 1.2]{Tou08}, \cite[Chapter 1]{Druel16}, \cite[Lemma 5.9]{Druel16I}, \cite[Section 1.1.3]{PT13}}]
\label{vanishing_first}
Let $X$ be a smooth projective variety and $\mathcal{G}$ be a regular codimension 1 foliation on $X$ with $c_1(\mathcal{G})=0$.
Then one of the following holds.
\begin{enumerate}
\item $X$ is a $\mathbb{CP}^1$-bundle over a smooth projective
variety $Y$ with $K_Y \equiv 0$ and $\mathcal{G}$ is everywhere transverse to the fiber of this $\mathbb{CP}^1$-bundle.
\item There exist an Abelian variety $A$, a simply connected smooth projective variety $Y$ with $K_Y \equiv 0$, and a finite \'etale cover $\pi: A \times Y \rightarrow X$, such that $\pi^{-1}\mathcal{G}$ is a pullback of a codimension 1 linear foliation on $A$.
\item There exist a smooth projective curve $B$ of genus at least 2, a smooth projective variety $Y$ with $K_Y \equiv 0$, and a finite \'etale cover $\pi: B \times Y \rightarrow X$, such that $\pi^{-1}\mathcal{G}$ is induced by the projection $B\times Y \rightarrow B$.
\end{enumerate}
\end{thm}

\begin{prop}
If ${\rm rk} \mathcal{F} =1$ and $c_1(\mathcal{F}) =0$, then one of the followings holds.
\begin{enumerate}
\item $S$ is a ruled surface over an elliptic curve $C$ and $\mathcal{F}$ is everywhere transverse to the fiber.
In particular, $\mathcal{F}$ is a Riccati foliation and induced by the following group representation
$$
\rho: \pi_1(C) = \mathbb{Z}^2 \rightarrow Aut(\mathbb{CP}^1)= PGL(2,\mathbb{C}).
$$
\item Up to a finite \'etale cover, $S$ is an Abelian variety and $\mathcal{F}$ is a linear foliation on $S$.
\item Up to a finite \'etale cover, $S \cong B \times C$ for some smooth projective curve $B$ of genus at least 2 and some elliptic curve $C$, and $\mathcal{F}$ is induced by the projection $B\times C \rightarrow B$.
\end{enumerate}
\end{prop}
\begin{proof}
It follows by Theorem \ref{vanishing_first} and \cite[Chapter 4.1]{Bru}.
\end{proof}

\subsubsection{Case of ${\rm rk} \mathcal{F} = 2$}
\text{}

In this case, $\mathcal{F} = T_S$.
Even if $T_S$ is pseudo-effective, the classification of the surface $S$ is not completed.
For example, we do not know whether a tangent bundle of a blow-up of a Hirzebruch surface along the general four points is pseudo-effective (see also \cite[Chapter 4]{HIM}).
If $S$ is a minimal surface, we can classify $S$ as below.

\begin{prop}[{\textit{cf.} \cite[Proposition 6.19]{DPS01}}]
Let $S$ be a minimal projective surface.
$T_S$ is almost nef if and only if $S$ belongs to the following list:

\begin{enumerate}
\item Up to a finite \'etale cover, $S$ is an Abelian variety.
\item $S$ is a minimal ruled surface over either $\mathbb{CP}^1$ or an elliptic curve. 
\item $S$ is isomorphic to $\mathbb{CP}^2$.
\end{enumerate}
\end{prop}

\begin{proof}

By \cite[Theorem 1.29]{KM98}, $S$ satisfies exactly one of the following conditions:

\hspace{-12pt}(Case 1). $K_S$ is nef. 
 
 \hspace{-12pt}(Case 2). $S$ is a minimal ruled surface over a curve. 
 
 \hspace{-12pt}(Case 3). $S \cong \mathbb{CP}^2$.
 
We classify $S$ in case 1 and 2.

\hspace{-12pt}(Case 1). 
Since $T_S$ is almost nef, by Theorem \ref{almost_nef}, $T_S$ is numerically flat.
By Yau's Theorem, $S$ is an Abelian variety up to a finite \'etale cover.

\hspace{-12pt}(Case 2). 
Let $\pi: S  \rightarrow C$ be a ruled surface over a curve $C$.
Then we have 
 \begin{equation*}
\xymatrix@C=25pt@R=20pt{
0\ar@{->}[r]&T_{S/C} = -K_{S/C} \ar@{->}[r]&T_{S} \ar@{->}[r]&\pi^{*}T_{C}\ar@{->}[r]&0. \\
}
\end{equation*}
By \cite[Proposition 3.8]{LOY20} and Corollary \ref{psef_ruled_surface}, $T_S$ is almost nef if and only if $\pi^{*}T_{C}$ is almost nef.
Hence $T_S$ is almost nef if and only if $C$ is $\mathbb{CP}^1$ or an elliptic curve.
\end{proof}

\section{Foliations with nef anti-canonical bundles}
In this chapter, let $X$ be a smooth $n$-dimensional projective variety.
We study a foliation with a nef anti-canonical bundle and give an answer to Question \ref{druel_conjecture}. For the proof, we use the following theorem.

\begin{thm}\cite[Claim 4.3]{Dsemi}
\label{Druel_nef}
Let $\mathcal{F}$ be a foliation on $X$.
Assume that $\mathcal{F}$ is regular or $\mathcal{F}$ has a compact leaf.
If $-K_{\mathcal{F}}$ is nef, then $\mathcal{F}$ contains an algebraically integrable foliation $\mathcal{G}$ with rationally connected leaves such that $\mathcal{G}$ has a compact leaf and $K_{\mathcal{F} }\equiv K_{\mathcal{G}}$.

\end{thm}

The following theorem is a consequence of \cite[Corollary 4.5]{Druel}.

\begin{thm}\cite[Corollary 4.5]{Druel}
\label{weakly_positivity}
Let $f:  Z \rightarrow Y$ be an equidimensional dominant morphism with connected fibers between normal $\mathbb{Q}$-factorial projective varieties, $\Delta$ be an effective $\mathbb{Q}$-divisor on $Z$, $L$ be a nef $\mathbb{Q}$-divisor on $Z$, and $F$ be a general fiber of $f$.
Assume that $(K_{Z/Y}+\Delta+L)|_{F}$ is pseudo-effective, $(F, \Delta |_{F})$ is lc, and $K_{Z} + \Delta +L$ is $\mathbb{Q}$-Cartier in a neighborhood of $F$.
Then $K_{Z/Y}+\Delta +L -Ram(f)$ is pseudo-effective.
\end{thm}

Let $\nu(X, L)$ denote the numerical dimension of $L$ for any nef $\mathbb{Q}$-divisor $L$ on $X$.
About the numerical dimension of a nef anti-canonical bundle of an algebraically integrable foliation, 
we have the following proposition.

\begin{prop}[{\textit{cf.} \cite[Theorem 3.9]{EIM}}]
\label{eim}
Let $\mathcal{F} $ be an algebraically integrable foliation on $X$. Assume that $\mathcal{F}$ is regular or $\mathcal{F} $ has a compact leaf.
If $-K_{\mathcal{F}}$ is nef, then
$\nu(F,-K_F) = \nu(X, -K_{\mathcal{F}})$ for a general leaf $F$ of $\mathcal{F}$.
\end{prop}

\begin{proof}
The basic idea is the same as \cite[Theorem 3.9]{EIM} and \cite[Theorem 4.2]{EG19}.
Notice that a general leaf $F$ of $\mathcal{F}$ is a smooth projective variety.
From $-K_{\mathcal{F}} |_{F} = -K_F$, we have 
$\nu(F,-K_F) \le \nu(X, -K_{\mathcal{F}})$.
We will show that $\nu(F,-K_F) \ge \nu(X, -K_{\mathcal{F}})$.

By applying Lemma \ref{algebraic} to the foliation $\mathcal{F}$,
there is an equidimensonal dominant morphism $f : Z \rightarrow Y$ with connected fibers between normal projective varieties and a birational morphism $\pi : Z \rightarrow X$.
Set $g:= f \circ \pi^{-1}: X \dashrightarrow Y$.
Since $\mathcal{F}$ is algebraically integrable, $g$ is almost holomorphic.
Hence we can take a regular point $y_0 \in Y \setminus f(Exc(\pi) \cup Ram(f))$. 
Set $F_0 := g^{-1} (y_0)$. Let $A$ be an ample divisor on $X$.
\begin{claim}
\label{ejiri}
There exists a constant $C>0$ such that $\text{mult}_{F_0}(\Gamma) \le C$ for any $m \in \mathbb{N}_{>0}$ and 
 any effective divisor $\Gamma$ with $\Gamma \equiv A-mK_{\mathcal{F}}$.
\end{claim}
\begin{proof}
Set $H_0 := f^{-1} (y_0)$.
Let $\sigma :  Z^{\flat} \rightarrow Z$ 
be the blow-up of $Z$ along $H_0$, $\rho:  Y^{\flat} \rightarrow Y$ be the blow up of $Y$ at $y_0$, 
and 
$H^{\flat}$ (resp. $G^{\flat}$) be the exceptional divisor of $\sigma$ (resp. $\rho$).
Let $f^{\flat} : Z^{\flat} \rightarrow Y^{\flat}$ be an induced morphism. Set $\pi^{\flat} := \pi \circ \sigma$.
(If ${ \rm codim} F_0 =1$, we assume that both $\sigma$ and $\rho$ are identity maps, $H^{\flat} = H_0$, and $G^{\flat} = y_0$.)
We now have the following commutative diagram:
\begin{equation*}
\xymatrix@C=25pt@R=20pt{
&Z^{\flat}\ar@{->}[d]^{\sigma }\ar@{->}[r]^{f^{\flat}}\ar@{->}[ld]_{\pi^{\flat} }& Y^{\flat} \ar@{->}[d]^{\rho }\\   
X \ar@/_15pt/@{-->}[rr]_{g}& Z \ar@{->}[r]^{f}\ar@{->}[l]_{\pi}&Y\\
}
\end{equation*}
Since $f$ is flat at every point in $H_0$, the square in the above diagram is cartesian. 
By \cite[Proposition B.20]{Wang}, there exists a $\pi^{\flat}$-exceptional effective $\mathbb{Q}$-divisor $\Delta^{\flat}$ on $Z^{\flat}$ such that 
$$
{\pi^{\flat}}^{*}K_{\mathcal{F}} \sim_{\mathbb{Q}} K_{Z^{\flat}/Y^{\flat}} -Ram(f^{\flat}) + \Delta^{\flat}.
$$
Since $-\Delta^{\flat} \sim_{\mathbb{Q}} 
 K_{Z^{\flat}/Y^{\flat}} + {\pi^{\flat}}^{*}(-K_{\mathcal{F}})-Ram(f^{\flat}) $ and 
$-K_{\mathcal{F}} $ is nef, $-\Delta^{\flat}$ is pseudo-effective by Theorem \ref{weakly_positivity}, and finally $\Delta^{\flat}=0$.

Notice that $\text{mult}_{F_0}(\Gamma) 
= \text{mult}_{H^{\flat}}({\pi^{\flat}}^{*}\Gamma)
= \text{mult}_{{f^{\flat}}^{*} G^{\flat} }({\pi^{\flat}}^{*}\Gamma)$.
Write $\alpha := \text{mult}_{{f^{\flat}}^{*} G^{\flat} }({\pi^{\flat}}^{*}\Gamma)$.
Then  ${\pi^{\flat}}^{*}\Gamma$ can be written as  ${\pi^{\flat}}^{*}\Gamma \equiv \alpha {f^{\flat}}^{*} G^{\flat} + E$ for some effective divisor $E$ with
$\text{mult}_{{f^{\flat}}^{*} G^{\flat} }E =0$.
We take  $p \gg 0$ such that $(Z^{\flat}_{y},\frac{E}{m+p}|_{Z^{\flat}_{y}})$ is lc for a general point $y \in Y^{\flat}$. 
From
$$
(m+p)K_{Z^{\flat} / Y^{\flat}} + E -p{\pi^{\flat}}^{*}K_{\mathcal{F}}
\equiv 
{\pi^{\flat}}^{*}A - \alpha {f^{\flat}}^{*} G^{\flat} + (m+p)Ram(f^{\flat}), 
$$
${\pi^{\flat}}^{*}A - \alpha {f^{\flat}}^{*} G^{\flat}$ is pseudo-effective by Theorem \ref{weakly_positivity}.
Let $A^{\flat}$ be an ample divisor on $Z^{\flat}$ and
$C :=\frac{({\pi^{\flat}}^{*}A){A^{\flat}}^{n-1}}{ ({f^{\flat}}^{*} G^{\flat}){A^{\flat}}^{n-1}}.$
Then we obtain $\alpha \le C$. 
\end{proof}

For any $k \in \mathbb{N}_{>0}$, we have 
 \begin{equation*}
\xymatrix@C=25pt@R=20pt{
0\ar@{->}[r]&\mathcal{I}_{F_0}^{k}  \ar@{->}[r]&\mathcal{I}_{F_0}^{k-1}  \ar@{->}[r]&\mathcal{I}_{F_0}^{k-1} /\mathcal{I}_{F_0}^{k}\ar@{->}[r]&0 .\\
}
\end{equation*}
By Claim \ref{ejiri} and the exact sequence above,
$$
h^{0}(X,A-mK_{\mathcal{F}} ) 
\le  \sum_{ 1 \le k \le \lfloor C \rfloor + 1} h^0(F_0, (A|_{F_0}-mK_{F_0} ) \otimes \mathcal{I}_{F_0}^{k-1} /\mathcal{I}_{F_0}^{k})
$$
for any $m \in \mathbb{N}_{>0}$.
Since $ \mathcal{I}_{F_0}^{k-1} /\mathcal{I}_{F_0}^{k}$ is a torsion-free $\mathcal{O}_{F_0}$-module, 
$ \mathcal{I}_{F_0}^{k-1} /\mathcal{I}_{F_0}^{k}$ is isomorphic to a subsheaf of the direct sum of $A|_{F_0}$.
Set $\nu := \nu(F_0,-K_{F_0})$, then 
there exists a constant $C'$ such that  
$$h^{0}(X,A-mK_{\mathcal{F}} ) / m^{\nu}
\le  \sum_{ 1 \le k \le \lfloor C \rfloor + 1} h^0(F_0, (A|_{F_0}-mK_{F_0} ) \otimes \mathcal{I}_{F_0}^{k-1} /\mathcal{I}_{F_0}^{k} )/ m^{\nu}
< C'
$$
for any $m \in \mathbb{N}_{>0}$, and finally $\nu(F,-K_F) =  \nu \ge \nu(X, -K_{\mathcal{F}})$.
\end{proof}

\begin{thm}[{\textit{cf.} \cite[Question 7.4]{Druel}}]
\label{rcfoliation}
Let $\mathcal{F} $ be a foliation on $X$. 
Assume that $\mathcal{F} $ is regular or $\mathcal{F} $ has a compact leaf. 
If $-K_{\mathcal{F}}$ is nef, then $\kappa(X, -K_{\mathcal{F}}) \le { \rm rk} \mathcal{F} $.
When the equality holds, the followings hold.
\begin{enumerate}
\item $\mathcal{F}$ is an algebraically integrable foliation with rationally connected leaves.
\item $-K_{F}$ is semiample and big for a general leaf $F$ of $\mathcal{F}$.
\item $-K_{\mathcal{F}}$ is semiample and abundant.
\end{enumerate}
\end{thm}
\begin{proof}
 By Theorem \ref{Druel_nef}, there exists an algebraically integrable foliation with rationally connected leaves $\mathcal{G} \subset \mathcal{F}$ with $K_{\mathcal{F}} \equiv K_{\mathcal{G}}$.
By Proposition \ref{eim}, 
for a general leaf $G$ of $\mathcal{G} $, 
$$
\kappa(X, -K_{\mathcal{F}}) \le \nu(X, -K_{\mathcal{F}})
 = \nu(X, -K_{\mathcal{G}})
  = \nu(G, -K_G)
  \le {\rm rk} \mathcal{G}
 \le {\rm rk} \mathcal{F}.
$$
When the equality holds, $-K_{\mathcal{F}}$ is nef and abundant, $ \mathcal{F} = \mathcal{G}$, and $\nu(X, -K_{\mathcal{F}}) = \nu(F, -K_F) =\dim F$ for a general leaf $F$ of $\mathcal{F}$.
By \cite[Corollary 8.4]{Druel},  
$-K_{\mathcal{F}}$ is semiample.
\end{proof}

\begin{cor}
Let $\mathcal{F} $ be a regular foliation on $X$.
If $-K_{\mathcal{F}}$ is nef and $\kappa(X, -K_{\mathcal{F}}) = {\rm rk} \mathcal{F} $, then there exists a smooth morphism $f:X \rightarrow Y$ such that $\mathcal{F}  = T_{X/Y}$.
\end{cor}
\begin{proof}
By Theorem \ref{rcfoliation}, a general leaf of $\mathcal{F}$ is rationally connected.
Therefore it follows  by \cite[Corollary 2.11]{Hor07}
\end{proof}

\begin{cor}
Let $f:X\rightarrow Y$ be a surjective morphism with connected fibers between smooth projective varieties.
If $-K_{X/Y}$ is nef and $\kappa(X, -K_{X/Y}) =\dim F$ for a general fiber $F$ of $f$, then there exists
a finite \'{e}tale morphism 
$\sigma:Y'\rightarrow Y$ such that $X \times_{Y} Y' \cong F \times Y'$.
\end{cor}
\begin{proof}
Let $\mathcal{F}$ be a foliation induced by $f$.
By Theorem \ref{rcfoliation} and \cite[Proposition 4.1]{EIM}, 
$K_{X/Y} = K_{\mathcal{F}}$ and $-K_{X/Y}$ is semiample.
Therefore the corollary follows from \cite[Proposition 4.4 and Theorem 4.7]{Amb}. 
\end{proof}

\begin{ex}\cite[Example 6.1]{EIM}
Let $C$ be an elliptic curve and  
$D$ be a divisor on $C$ with $deg(D)=0$, 
$D \not\sim 0$, and $2D\sim 0$. 
Set $E : = \mathcal{O}_{C} \oplus \mathcal{O}_{C}(D)$ and $X:= \mathbb{P}(E)$.
Let $f : X \rightarrow C$ be the ruling.
Then $-K_{X/C} \sim \mathcal{O}_{\mathbb{P}(E)}(2)\otimes f^{*} \mathcal{O}_{C}(-D)$
and $-K_{X/C}$ is nef.
For any $m \in \mathbb{N}_{>0}$, 
\begin{align*}
\begin{split}
h^{0}(X,-mK_{X/C} )& = h^0(C,\Sym^{2m}(\mathcal{O}_{C} \oplus \mathcal{O}_{C}(D)) \otimes\mathcal{O}_{C}(-mD) ) \\
&=\sum_{i=-m}^{m}h^{0}(C,\mathcal{O}_{C}(iD) ) =2\lfloor m/2 \rfloor+1.
\end{split}
\end{align*}
We thus get $\kappa(X,-K_{X/C}) = 1 = \dim X -1$.
However $X$ is not isomorphic to  $C \times \mathbb{P}^{1}$. 
This gives a counter-example of \cite[Question 8.11]{Druel}.

\end{ex}

Finally, we propose a conjecture on a foliation with a nef anti-canonical bundle.
Even if $-K_X$ is nef, it is expected that the same thing as Corollary \ref{psef_example} holds.

\begin{conj}
\label{MRC_conj}
Let $X$ be a smooth $n$-dimensional projective variety with a nef anti-canonical bundle and 
$\mathcal{F}$ be a foliation on $X$.
Then $\mathcal{F}$ is induced by the MRC fibration of $X$
if and only if 
$\mu_{\alpha}^{min}(\mathcal{F}) >0$ and $K_{\mathcal{F}} \equiv K_{X} $, where $\alpha = H_1\cdots H_{n-1}$ for some ample divisors $H_1, \ldots, H_{n-1}$ on $X$.
\end{conj}

\renewcommand{\thesection}{\Alph{section}} \setcounter{section}{0}
\section{Appendix} \label{ruled_surface}
In this appendix, we study divisors of a ruled surface.
For more details, we refer the reader to \cite[Chapter 5.2]{Har}.
Let $C$ be a smooth curve, $E$ be a rank 2 vector bundle, $\pi : X:= \mathbb{P}(E) \rightarrow C$ be the ruled surface on $C$, and $F$ be a fiber of $\pi$.
By \cite[Proposition 2.8 and Notation 2.8.1]{Har}, we may assume that
$H^{0}(C,E) \neq 0$ and $H^{0}(C,E\otimes L) =0$ for any line bundle $L$ on $C$ with $\deg L <0$.
Then the N\'eron-Severi group $NS(X)$ is generated by $\xi := c_1(\mathcal{O}_{\mathbb{P}(E)}(1))$ and $f := c_1(F)$.
Set $e := -\deg E$, then $\xi^2=\deg(E)=-e$, $\xi f=1$, and $f^2=0$.

\begin{prop}\cite[Proposition 2.20 and Proposition 2.21]{Har}
\label{ample_div}
Let $D$ be a divisor on $X$ with $c_1(D) = a\xi + bf$.
\begin{enumerate}
\item If $e \ge 0$, then $D$ is ample if and only if $a>0$ and $b>ae$.
\item If $e<0$, then $D$ is ample if and only if $a>0$ and $2b>ae$.
\end{enumerate}
\end{prop}

Since $X$ is a surface, the movable cone of $X$ is a closure of the ample cone of $X$, hence we have the following proposition.

\begin{prop}
\label{psef_div}
Let $D$ be a divisor on $X$ with $c_1(D) = a\xi + bf$.
\begin{enumerate}
\item If $e \ge 0$, then $D$ is pseudo-effective if and only if $a\ge 0$ and $b\ge0$.
\item If $e<0$, then $D$ is pseudo-effective if and only if $a\ge0$ and $2b \ge ae$.
\end{enumerate}
\end{prop}

\begin{cor}
\label{psef_ruled_surface}
$-K_{X/C}$ is nef if and only if $e\le0$, and $-K_{X/C}$ is big if and only if $e>0$.
In particular, $-K_{X/C}$ is pseudo-effective.
\end{cor}
\begin{proof}
From $-K_{X/C} \sim \mathcal{O}_{\mathbb{P}(E)}(2)\otimes \pi^{*}(\det E^{\vee})$, we have $c_1(-K_{X/C}) = 2\xi + ef$, which completes the proof by Proposition \ref{ample_div} and \ref{psef_div}.
\end{proof}

\bibliographystyle{alpha}

\end{document}